\documentclass[12pt]{amsart}

\usepackage{tikz-cd}

\usepackage{amssymb}
\usepackage{amsthm}
\usepackage{amscd}

\usepackage[mathscr]{eucal}
\usepackage{enumitem}

\usepackage{hyperref}

\usepackage{color}
\newtheorem*{Main Theorem}{Main Theorem}

\newtheorem{Theorem}{Theorem}[section]
\newtheorem{theorem}[Theorem]{Theorem}

\newtheorem{corollary}[Theorem]{Corollary}

\newtheorem{lemma}[Theorem]{Lemma}

\newtheorem{fact}[Theorem]{Fact}
\newtheorem{remark/def}[Theorem]{Remark/Definition}

\theoremstyle{definition}

\newtheorem{example}[Theorem]{Example}

\newtheorem{remark}[Theorem]{Remark}

\newtheorem{definition}[Theorem]{Definition}

\newtheorem{def/rem}[Theorem]{Definition/Remark}
\newtheorem{def/fact}[Theorem]{Definition/Fact}
\newtheorem{not/rem}[Theorem]{Notation/Remark}
\newsavebox{\indbin}
\savebox{\indbin}{\begin{picture}(0,0)
\newlength{\gnu}
\settowidth{\gnu}{$\smile$} \setlength{\unitlength}{.5\gnu}
\put(-1,-.65){$\smile$} \put(-.25,.1){$|$}
\end{picture}}

\newcommand{\be}{\begin{enumerate}}
\newcommand{\bi}{\begin{itemize}}

\newcommand{\bd}{\begin{defn}}

\newcommand{\bt}{\begin{theorem}}
\newcommand{\bl}{\begin{lemma}}
\newcommand{\ee}{\end{enumerate}}
\newcommand{\ei}{\end{itemize}}
\newcommand{\ed}{\end{defn}}
\newcommand{\et}{\end{theorem}}
\newcommand{\el}{\end{lemma}}

\def\fm{\mathfrak{m}}

\newcommand{\Th}{\operatorname{Th}}

\def\res{\operatorname{res}}

\def\Im{\operatorname{Im}}

\def\H{\mathcal H}

\def\hp{+_{\mathcal H}}
\def\hm{\cdot_{\mathcal H}}
\def\hn{\nu_{\mathcal H}}
\def\hsum{\sum^{\mathcal H}}

\title[The AKE principles via the higher valued hyperfield]{The Ax-Kochen-Ershov principles via the higher valued hyperfield}
\author{Junguk Lee}
\address{Department of Mathematics, Changwon National University\\ Changwon 51140\\ South Korea}
\email{ljwhayo@changwon.ac.kr}

\begin{document}

\begin{abstract}
In this paper, we concern the model theory of finitely ramified henselian valued fields via higher valued hyperfields.

Most of all, we provide a number of Ax-Kochen-Ershov Theorems for finitely ramified henselian valued fields relative to higher valued hyperfields. As corollaries, we deduce a transfer of decidability for full theories and existential theories of a finitely ramified henselian valued fields relative to higher valued hyperfields. 
\end{abstract}

\maketitle

\section{Introduction}
In this paper, we concern model theory of finitely ramified henselian valued fields via higher valued hyperfields.
% Mainly, we aim to provide several versions of Ax-Kochen-Ershov Theorems for finitely ramified henselian valued fields relative to higher valued hyperfields.
% We will prove elementary equivalence, existential equivalence, elementarity of substructures, and existential closedness for theories of finitely ramified henselianvalued fields relative to valued higher hyperfields.

A valued field $(K,\nu)$ of mixed characteristic $(0,p)$ is called {\em finitely ramified} if the interval $(0,\nu(p)]$ in the value group is finite, and its cardinality $e$ is called the {\em initial ramification index}. If the initial ramification index $e$ is one, then the valued fields is called {\em unramified}. Let $\mathcal O_{\nu}$ be the valuation ring of $K$ with the maximal ideal $\mathfrak{m}_{\nu}$. Let $\nu K$ be the value group and we say $(K,\nu)$ is $\mathbb Z$-valued field if $\nu K\cong \mathbb Z$. Let $K\nu$ be the residue field of $K$ and $\res_{\nu}:\mathcal O_{\nu}\rightarrow K\nu$ be the natural projection map. For each $n\ge 1$, let $\mathcal O_{\nu,n}$ be the quotient of $\mathcal O_{\nu,n}$ by $\mathfrak{m}_{\nu}^n$.\\
% If there is no confusion, we omit $\nu$ in the subscript.

In \cite{ADJ24}, Anscombe, Dittmann, and Jahnke proved completeness, existential completeness, and model completeness results for theories of finitely ramified henselian valued fields relative to residue fields and value groups.
\begin{fact}\label{fact:main_ADJ24}\cite[Theorem 5.1, Corollary 5.6, Theorem 5.13]{ADJ24}
If $(K,\nu)$ and $(L,\omega)$ are finitely ramified henselian valued fields of the same initial ramification index, then
$$\underbrace{(K,\nu)\equiv (L,\omega)}_{\mbox{ in }\mathcal L_{val}}\Leftrightarrow \underbrace{K\nu\equiv L\omega}_{\mbox{ in }\mathcal L_{p,e}}\mbox{ and }\underbrace{\nu K\equiv \omega L}_{\mbox{ in }\mathcal L_{oag}}$$ and
$$\underbrace{(K,\nu)\equiv_{\exists} (L,\omega)}_{\mbox{ in }\mathcal L_{val}}\Leftrightarrow \underbrace{K\nu\equiv_{\exists} L\omega}_{\mbox{ in }\mathcal L_{p,e}}$$ and,
in case $(K,\nu)\subseteq (L,\omega)$,
$$\underbrace{(K,\nu)\preceq_{\exists} (L,\omega)}_{\mbox{ in }\mathcal L_{val}}\Leftrightarrow \underbrace{K\nu\preceq_{\exists} L\omega}_{\mbox{ in }\mathcal L_{p,e}}\mbox{ and }\underbrace{\nu K\preceq_{\exists} \omega L}_{\mbox{ in }\mathcal L_{oag}},$$
where $\mathcal L_{val}$ is the three-sorted language of valued fields, $\mathcal L_{oag}$ the language of ordered abelian groups, and $\mathcal L_{p,e}$ is an expansion of the ring language $\mathcal L_{ring}$ by adding some finitely many relation symbols, which are existentially definable in $\mathcal L_{val}$ (see the first paragraph of Section \ref{sect:preliminaries} for precise definitions of the languages $\mathcal L_{val}$, $\mathcal L_{oag}$, and $\mathcal L_{ring}$).
\end{fact}
Most of all, their AKE principle for elementary equivalence strengths the results of Basarab and of Lee and the author. In \cite[Theorem 3.1]{Bas78}, Basarab proved that
$$\underbrace{(K,\nu)\equiv (L,\omega)}_{\mbox{ in }\mathcal L_{val}}\Leftrightarrow \underbrace{\mathcal O_{\nu,n}\equiv \mathcal O_{\omega,n}}_{\mbox{ in }\mathcal L_{ring}}\mbox{ for all }n\ge 1\mbox{ and }\underbrace{\nu K\equiv \omega L}_{\mbox{ in }\mathcal L_{oag}}.$$
If residue fields are perfect, we in \cite[Theorem 5.2]{LL21} showed that it is enough to consider a single $n=n_0$ depending only on $p$ and $e$, not considering all $n$ simultaneously.\\

Meanwhile, we in \cite{Lee20} proved the AKE principle for elementary equivalence relative to higher valued hyperfields when residue fields are perfect. Given a valued field $(K,\nu)$ and a positive integer $n\ge 1$, the {\em $n$th valued hyperfield} is $\mathcal H_{\nu,n}:=K/(1+\mathfrak{m}_{\nu}^n)$. The higher valued hyperfield forms a certain valued field-like algebraic structure with a multivalued addition. We consider higher valued hyperfields as structures in the one-sorted language $\mathcal L_{vhf}$ of valued hyperfields (see Subsection \ref{subsect:valued hyperfield} for the detailed definition of the language). Given integers $n$ and $m$, we write $n|m$ if $n$ divides $m$.
\begin{fact}\label{fact:main_Lee20}\cite[Theorem 5.8, Corollary 5.9]{Lee20}
Let $(K,\nu)$ and $(L,\omega)$ be finitely ramified henselian valued fields of mixed characteristic $(0,p)$ having the same initial ramification index $e$. Let $n$ be a positive integer satisfying that
$$\begin{cases}
n>0 &\mbox{if } p\not| e,\\
n>e^2M_{\nu}&\mbox{if } p|e. 
\end{cases}
$$
Then,
$$\underbrace{(K,\nu)\equiv (L,\omega)}_{\mbox{ in } \mathcal L_{val}}\Leftrightarrow \underbrace{(\mathcal H_{\nu,n},[p]_n)\equiv (\mathcal H_{\omega,n},[p]_n)}_{\mbox{ in }\mathcal L_{vhf}}.$$ 
\end{fact}
Our main results in this paper are several versions of Ax-Kochen-Ershov principles for finitely ramified henselian valued fields relative to higher valued hyperfields without any restriction on residue fields, analogous to Fact \ref{fact:main_ADJ24} and generalizing Fact \ref{fact:main_Lee20}.
\begin{Main Theorem}[Theorem \ref{theorem:relative completeness}, Theorem \ref{theorem:relative existential completeness}, Theorem \ref{theorem:relative existential closedness}]
Let $(K,\nu)$ and $(L,\omega)$ be finitely ramified henselian valued fields of mixed characteristic $(0,p)$ having the same initial ramification index $e$. Let $n$ be a positive integer satisfying that
$$\begin{cases}
n>0 &\mbox{if } p\not| e,\\
n>e^2M_{\nu}&\mbox{if } p|e. 
\end{cases}
$$
Then,
$$\underbrace{(K,\nu)\equiv (L,\omega)}_{\mbox{ in }\mathcal L_{val}}\Leftrightarrow \underbrace{(\mathcal H_{\nu,n},[p]_n)\equiv (\mathcal H_{\omega,n},[p]_n)}_{\mbox{ in }\mathcal L_{vhf}}$$ and
$$\underbrace{(K,\nu)\equiv_{\exists} (L,\omega)}_{\mbox{ in }\mathcal L_{val}}\Leftrightarrow \underbrace{(\mathcal H_{\nu,n},[p])\equiv_{\exists} (\mathcal H_{\omega,n},[p]_n)}_{\mbox{ in }\mathcal L_{vhf}}$$ and,
in case $(K,\nu)\subseteq (L,\omega)$,
$$\underbrace{(K,\nu)\preceq_{\exists} (L,\omega)}_{\mbox{ in }\mathcal L_{val}}\Leftrightarrow \underbrace{\mathcal H_{\nu,n}\preceq_{\exists} \mathcal H_{\omega,n}}_{\mbox{ in }\mathcal L_{vhf}}.$$
\end{Main Theorem}
%The main theorem is given (as Theorem \ref{theorem:relative completeness}, Theorem \ref{theorem:relative existential completeness}, and Theorem \ref{theorem:relative existential closedness}) in Section \ref{sec:AKE via hyperfields}.
As corollaries, we deduce a transfer of decidability for full theories (see Corollary \ref{cor:transfer of decidability of full theories}) and existential theories (see Corollary \ref{cor:Hilbert's 10th via the higher valued field}). Note that Kartas proved that for henselian valued fields $(K,\nu)$ and $(L,\omega)$ tamely (and not necessarily finitely) ramified extensions of $(\mathbb Q,\nu_p)$ with the $p$-adic valuation $\nu_p$,
$$\underbrace{(\mathcal H_{\nu,1},[p]_1)\equiv_{\exists} (\mathcal H_{\omega,1},[p]_1)}_{\mathcal L_{vhf}}\Rightarrow \underbrace{(K,\nu)\equiv_{\exists} (L,\omega)}_{\mathcal L_{ring}}$$ under a certain form of resolution of singularities (see \cite[Theorem A]{Kar23}).\\

Now we briefly explain the structure of the paper. In Section \ref{sect:preliminaries}, we  remind several facts on basic terminologies used throughout the paper. Most of all, we concern several facts on the strict Cohen subring and its representatives developed by Anscombe and Jahnke in \cite{AJ22}. Also, we consider basic properties on valued hyperfields. And we introduce the one-sorted language $\mathcal L_{vhf}$ of valued hyperfields and concern the $\mathcal L_{vhf}^+$-homomorphism of valued hyperfields, which is crucially used for the AKE principles for existentially completeness and existentially closedness relative to higher valued hyperfields (see Remark \ref{rem:L_vhf homomorphism and vhf homomorphism}). In Section \ref{sect:lifting via valued hyperfield}, we prove two versions of embedding theorems of valued fields relative to higher valued hyperfields for complete $\mathbb Z$-valued fields: over structures and not over structures (see Theorem \ref{theorem:lifting lemma} and Lemma \ref{lem:embedding over valued fields}). Finally, in Section \ref{sect:AKE via hyperfields}, we apply our embedding theorems to prove the AKE principles for finitely ramified valued fields relative to higher valued hyperfieds.

\section{Preliminaries}\label{sect:preliminaries}
% For each $i\ge 1$, let $R_{(i)}:=R/\mathfrak{m}^i$ be the quotient of $R$ by $\mathfrak{m}^i$, and for $i\ge j\ge 1$, $\res_i:R\rightarrow R_{(i)}$ and $\res_{i,j}:R_{(i)}\rightarrow R_{(j)}$ be the natural projection maps respectively. By abusing notation, we write $\res=\res_1:R\rightarrow k(\cong R/\mathfrak{m})$. 

Throughout the paper, we consider valued fields as first-order structures in the three-sorted language $\mathcal L_{val}$: A sort $K$ for the field, a sort $k$ for the residue field, and a sort $\Gamma\cup\{\infty\}$ for the value group with infinity. Two field sorts $K$ and $k$ are interpreted as rings in the language $\mathcal L_{ring}:=\{+,\cdot,0,1\}$ and the value group sort $\Gamma\cup\{\infty\}$ interpreted as an ordered abelian group with infinity in the language $\mathcal L_{oag}:=\{+,0,<,\infty\}$. And three sorts are connected by symbols for the residue map $\res:K\rightarrow k$, interpreted as the constant zero map outside the valuation ring, and the valuation map $\nu:K\rightarrow \Gamma\cup\{\infty\}$. 

\subsection{The strict Cohen subring and Representatives}\label{subsect:Cohen subring and its representatives}
In this subsection, we assume that $(K,\nu)$ is complete and $\mathbb Z$-valued. A {\em strict Cohen ring} is a complete unramified valuation ring of mixed characteristic $(0,p)$ (c.f. \cite[Section 2]{AJ22}). We recall that $\mathcal O_{\nu}$ has a strict Cohen subring over which $\mathcal O_{\nu}$ is totally ramified, and we call such a ring a {\em strict Cohen ring of $\mathcal O_{\nu}$}. For a field $k$ of characteristic $p>0$ and a subfield $k_0$ of $k$, a subset $B$ of $k$ is called {\em p-independent over $k_0$} if $$[k^p k_0(b_1,\ldots, b_r):k^p k_0]=p^r$$ for all distinct elements $b_1,\ldots,b_r\in B$ and for all $r>0$, and $B$ is called a {\em p-basis over $k_0$} if $k=k^pk_0(B)$. If $k_0$ is the prime field, then we omit `over $k_0$'.

\begin{fact}[{\cite[Theorem 11 and Theorem 13]{Coh46}}]\label{fact:existence of a strict Cohen subring}\
\begin{enumerate}
	\item Each tuple $B$ of representatives of a $p$-basis of $K\nu$ is contained a strict Cohen subring $R$ of $\mathcal O_{\nu}$.
	
	\item Given tuples $B$ and $B'$ of representatives of a same $p$-basis of $K\nu$ and strict Cohen subrings $R$ and $R'$ of $\mathcal O_{\nu}$ containing $B$ and $B'$ respectively, $B=B'$ (as a set) if and only if $R=R'$. Especially, if $K\nu$ is perfect or $\mathcal O_{\nu}$ is unramified, then such a strict Cohen subring is unique. So, given a tuple $B$ of representatives of a $p$-basis of $K\nu$, we write $R(B)$ for the strict Cohen subring of $\mathcal O_{\nu}$ containing $B$.
\end{enumerate}
\end{fact}

We also recall the following theorem on lifting of homomorphism of residue fields into a homomorphism of valued fields.
\begin{fact}[{\cite[Lemma 4.1]{ADJ24}}]\label{fact:lifting lemma}
Let $(K,\nu)$ and $(L,\omega)$ be $\mathbb Z$-valued fields of mixed characteristic $(0,p)$. Suppose $K$ is unramified and $L$ is complete. Let $\varphi:K\nu\rightarrow L\omega$ be a field embedding. Let $\beta$ be a tuple of $p$-independent elements in $K\nu$, $B$ a lift of $\beta$ in $K$, and $B'$ a lift of $\varphi(\beta)$ in $L$.

Then, there is an embedding $\Phi:K\rightarrow L$ of valued fields sending $B$ to $B'$, compatible with $\varphi$, that is, given $x\in \mathcal O_{\nu}$, $$\res_{\omega}(\Phi(x))=\varphi(\res_{\nu}(x)).$$
%\begin{enumerate}
%	\item Let $R_1$ and $R_2$ be strict Cohen rings with the residue fields $k_1$ and $k_2$. Let $\varphi:k_1\rightarrow k_2$ be a field embedding over $k_0$ such that $k_2/\varphi(k_1)$ is a separable extension. Let $\beta$ be a tuple of $p$-independent elements in $k_1$, $b$ a lift of $\beta$ in $R_1$, and $b'$ a lift of $\varphi(\beta)$ in $R_2$. Then, there is a unique embedding $\Phi:R_1\rightarrow R_2$ of valued fields sending $b$ to $b'$ and compatible with $\varphi$. Moreover, if $\varphi$ is an isomorphism, then $\Phi$ is an isomorphism.
%
%	
%	\item Let $K$ and $L$ be $\mathbb Z$-valued fields with residue fields $k$ and $l$. Suppose $K$ is unramified and $L$ is complete. Let $\varphi:k\rightarrow l$ be a field embedding. Let $\beta$ be a tuple of $p$-independent elements in $k$, $b$ a lift of $\beta$ in $K$, and $b'$ a lift of $\varphi(\beta)$ in $L$. Then, there is an embedding $\Phi:K\rightarrow L$ of valued fields sending $b$ to $b'$, compatible with $\varphi$.
%\end{enumerate}
\end{fact}

Now, we recall several facts on representatives of the residue field in the valuation ring.
\begin{fact}[c.f. {\cite[Lemma 6]{Coh46}}]\label{fact:uniquness of p-power representative}
For $a,b\in \mathcal O_{\nu}$,  $$a\equiv b\pmod{\mathfrak{m}_{\nu}}\Rightarrow a^{p^i}\equiv b^{p^i} \pmod{ \mathfrak{m}_{\nu}^i}$$ for each $i\ge 1$.
\end{fact}

\begin{fact}[c.f. {\cite[Lemma 7]{Coh46}} and {\cite[Lemma 3.4]{AJ22}}]\label{fact:uniqune representative}
There is a unique choice of representatives $$\lambda_{i+1}:(K\nu)^{(p^i)}\rightarrow \mathcal O_{\nu,i+1}^{(p^i)},\alpha^{p^i}\mapsto a^{p^i}+\mathfrak{m}_{\nu}^{i+1}$$ where $a\in R$ with $\alpha=\res(a)$ for each $i\ge 0$.
\end{fact}

\begin{remark}\label{rem:choice of respresentative in cohen subring}
For a strict Cohen subring $R$ of $\mathcal O_{\nu}$, the restriction $\res_{\nu}|_{R}:R\rightarrow K\nu$ is surjective and so we have $$\Im(\lambda_{l+1})=\{c^{p^l}\pmod{\mathfrak{m}^{l+1}}:c\in R\}.$$ 
\end{remark}

Given $m\ge 1$ and $l\ge 0$, let $$P_{m,l}:=\{(i_1,\ldots,i_m):0\le i_1,i_2,\ldots,i_m<p^l\}$$ be the set of all multi-indices in $m$-many non-negative integers less than $p^l$. Given $m$-tuple $\bar b:=(b_1,\ldots,b_m)$ in a ring and a multi-index $I=(i_1,\ldots,i_m)\in P_{m,l}$, we write $$\bar b^I:=b_1^{i_1}\cdots b_m^{i_m}.$$
\begin{fact}[c.f. {\cite[Proposition 3.6]{AJ22}}]\label{fact:generated by p-basis and representatives}
Fix $l\ge 0$. Let $\lambda_{l+1}:(K\nu)^{(p^l)}\rightarrow \mathcal O_{\nu,l+1}^{(p^l)}$ be the unique surjective choice of representatives and let $B$ be a tuple of representatives of a fixed $p$-basis of $K\nu$. Let $R$ be the strict Cohen subring of $\mathcal O_{\nu}$ containing $B$.
\begin{enumerate}
	\item $\mathcal O_{\nu,l+1}$ is generated by $\Im(\lambda_{l+1})\cup \res_{l+1}[B] \cup\{\res_{l+1}(\pi)\}$.
More precisely, for each $a\in \mathcal O_{\nu}$, there are
\begin{itemize}
	\item a tuple $\bar b:=(b_1,\ldots,b_m)\in B^m$ for some $m\ge 1$; and
	\item tuples $\bar \alpha_i:=(\alpha_{i,I})_{I\in P_{m,l}}\in (K\nu)^{|P_{m,l}|}$ for $i\le l$;
\end{itemize} 
such that $$a\equiv \sum_{i\le l} \left(\sum_{I\in P_{m,l}}\lambda_{l+1}(\alpha_{i,I})^{p^l} \bar b^I\right)\pi^i \pmod{\mathfrak{m}_{\nu}^{l+1}}.$$
Equivalently, there are tuples $\bar a_i:=(a_{i,I})_{I\in P_{m,l}}\in \mathcal O_{\nu}^{|P_{m,l}|}$ for $i\le l$ such that $$a\equiv \sum_{i\le l} \left(\sum_{I\in P_{m,l}}a_{i,I}^{p^l} \bar b^I\right)\pi^i \pmod{\mathfrak{m}_{\nu}^{l+1}}.$$

	\item The subring $R\pmod{\pi^{l+1}}$ of $\mathcal O_{\nu,l+1}$ is generated by $$\Im(\lambda_{l+1})\cup \res_{l+1}[B].$$ More precisely, let $s\ge 0$ with $\nu(p^s)<l+1\le \nu(p^{s+1})$, and then for each $a \in R$, there are
\begin{itemize}
	\item a tuple $\bar b:=(b_1,\ldots,b_m)\in B^m$ for some $m\ge 1$; and
	\item tuples $\bar \alpha_i:=(\alpha_{i,I})_{I\in P_{m,l}}\in (K\nu)^{|P_{m,l}|}$ for $i\le s$;
\end{itemize} 
such that $$a\equiv \sum_{i\le s} \left(\sum_{I\in P_{m,l}}\lambda_{l+1}(\alpha_{i,I})^{p^l} \bar b^I\right)p^i \pmod{\mathfrak{m}_{\nu}^{l+1}}.$$
Equivalently, there are tuples $\bar a_i:=(a_{i,I})_{I\in P_{m,l}}\in R^{|P_{m,l}|}$ for $i\le l$ such that $$a\equiv \sum_{i\le s} \left(\sum_{I\in P_{m,l}}a_{i,I}^{p^l} \bar b^I\right)p^i \pmod{\mathfrak{m}_{\nu}^{l+1}}.$$

\end{enumerate}
\end{fact}
\begin{proof}
We provide a proof of $(2)$. We choose inductively tuples $\bar a_{i}:=(a_{i,I})_{I\in P_{m,l}}\in R^{|P_{m,l}|}$ such that $$a-\sum_{i\le i_0}\left(\sum_{I\in P_{m,l}}a_{i,I}^{p^l}\bar b^I\right)p^i\in p^{i_0+1}R$$ for each $i_0\le s$. For $i_0=0$, by Remark \ref{rem:choice of respresentative in cohen subring}, there is a tuple $\bar a_0:=(a_{i,I})_{I\in P_{m,l}}\in R^{|P_{m,l}|}$ such that $$a=\sum_{I\in P_{m,l}}a_{0,I}^{p^l}\bar b^I\pmod{\mathfrak{m}},$$ and so $a-\sum_{I\in P_{m,l}}a_{0,I}^{p^l}\bar b^I\in \mathfrak{m}_{\nu}\cap R=pR$. Suppose we have chosen tuples $\bar a_{0},\ldots,\bar a_{i_0}$ such that $$a-\sum_{i\le i_0}\left(\sum_{I\in P_{m,l}}a_{i,I}^{p^l}\bar b^I\right)p^i\in p^{i_0+1}R$$ for $i_0<s$. Since $\left(a-\sum_{i\le i_0}\left(\sum_{I\in P_{m,l}}a_{i,I}^{p^l}\bar b^I\right)p^i\right)/p^{i_0+1}\in R$, by Remark \ref{rem:choice of respresentative in cohen subring}, there is a tuple $\bar a_{i_0+1}:=(a_{i_0+1,I})_{I\in P_{m,l}}\in R^{|P_{m,l}|}$ such that
\begin{align*}
&\left(a-\sum_{i\le i_0}\left(\sum_{I\in P_{m,l}}a_{i,I}^{p^l}\bar b^I\right)p^i\right)/p^{i_0+1}-\sum_{I\in P_{m,l}}a_{i_0+1,I}^{p^l}\bar b^I\in pR\\
\Leftrightarrow&\ a-\sum_{i\le i_0}\left(\sum_{I\in P_{m,l}}a_{i,I}^{p^l}\bar b^I\right)p^i-\left(\sum_{I\in P_{m,l}}a_{i_0+1,I}^{p^l}\bar b^I\right)p^{i_0+1}\in p^{(i_0+1)+1}R\\
\Leftrightarrow&\ a-\sum_{i\le i_0+1}\left(\sum_{I\in P_{m,l}}a_{i,I}^{p^l}\bar b^I\right)p^i\in p^{(i_0+1)+1}R
\end{align*}
and so the chosen tuples $\bar a_{\le i_0+1}$ are desired ones.
\end{proof}

\subsection{Valued hyperfields}\label{subsect:valued hyperfield}
First, we introduce a notion of valued hyperfields.
\begin{definition}\label{def:vhf}\cite[Definition 1.2 and 1.4]{Kra57}
\begin{enumerate}
\item A {\em hyeprfield} is an algebraic structure $(H,+,\cdot,0,1)$ such that for $H^{\times}:=H\setminus\{0\}$, $(H^{\times},\cdot,1)$ is an abelian group and there is a multivalued operation $+:\ H\times H\rightarrow 2^{H}$ for the power set $2^H$ of $H$ satisfying the following:
\begin{enumerate}
	\item $0\cdot \alpha=0$ for all $\alpha\in H$.
	\item (Associative) $(\alpha+\beta)+\gamma=\alpha+(\beta+\gamma)$ for all $\alpha,\beta,\gamma\in H$. 
	\item (Commutative) $\alpha+\beta=\beta+\alpha$ for all $\alpha,\beta\in H$.
	\item (Distributive) $(\alpha+\beta)\cdot \gamma\subset \alpha\cdot\gamma+\beta\cdot\gamma$ for all $\alpha,\beta,\gamma\in H$. 
	\item (Identity) $\alpha+0=\{\alpha\}$ for all $\alpha\in H$.
	\item (Inverse) For any $\alpha\in H$, there is a unique $-\alpha\in H$ such that $0\in \alpha+(-\alpha)$.
	\item For all $\alpha,\beta,\gamma\in H$, $\alpha\in \beta+(-\gamma)$ if and only if $\beta\in \alpha+\gamma$.	 
\end{enumerate}
	
	\item A {\em valued hyperfield} is a hyperfield $(H,+,\cdot,0,1)$ equipped with a map $\nu$ from $H$ to $\Gamma\cup\{\infty\}$ for an ordered abelian group $\Gamma$ such that 
	\begin{enumerate}
		\item For $\alpha\in H$, $\nu(\alpha)=\infty$ if and only if $\alpha=0$;
		\item For all $\alpha,\beta\in H$, $\nu(\alpha\cdot\beta)=\nu(\alpha)+\nu(\beta)$;
		\item For all $\alpha,\beta\in H$ and $\gamma\in \alpha+\beta$, $\nu(\gamma)\ge \min\{\nu(\alpha),\nu(\beta)\}$;
		\item For all $\alpha,\beta\in H$, $\nu(\alpha+\beta)$ consists of a single element unless $0\in \alpha+\beta$; and
		\item There is $\rho_H\in \Gamma$ such that either $\alpha+\beta$ is a closed ball of radius $\rho_H+\min\{\nu(\alpha),\nu(\beta)\}$ for all $\alpha,\beta\in H$, or $\alpha+\beta$ is a open ball of radius $\rho_H+\min\{\nu(\alpha),\nu(\beta)\}$ for all $\alpha,\beta\in H^{\times}$.
	\end{enumerate}
\end{enumerate}
For $B\subset H$ and $\alpha\in H$, define $\alpha+ B:=\bigcup_{\beta \in B}\alpha+ \beta$(*). The associativity of $+$ means that given $\alpha,\beta,\gamma\in H$, we have $(\alpha+\beta),(\beta+ \gamma)\subset H$ and $\alpha+(\beta+\gamma)=(\alpha+\beta)+\gamma$ in the sense of $(*)$.
\end{definition}
\noindent For $\alpha_0,\ldots, \alpha_k \in H$, we write $\hsum \alpha_i$ for $(\alpha_0+\cdots+ \alpha_k)\subset H$. Since the multivalued operation $+$ is associative, the notion of $\hsum$ is well-defined.
\begin{definition}\label{def:morphism_vhf}\cite[Definition 2.5]{Lee20}
Let $(H_i,+_i,\cdot_i,0_i,1_i,\nu_i)$ be a valued hyperfield for $i=1,2$. A map $f$ from $H_1$ to $H_2$ is called a {\em homomorphsim} if the following hold:
\begin{enumerate}
	\item $f(0_1)=0_2$ and $f(1_1)=1_2$.
	\item $f(\alpha\cdot_1 \beta)=f(\alpha)\cdot_2 f(\beta)$ for all $\alpha,\beta \in H_1$.
	\item $f(\alpha+_1\beta)\subset f(\alpha)+_2 f(\beta)$ for all $\alpha,\beta\in H_1$.
	\item For all $\alpha,\beta \in H_1$, $$\nu_1(\alpha)\le \nu_1(\beta)\Leftrightarrow\nu_2(f(\alpha_1))\le \nu_2(f(\beta)).$$
\end{enumerate}
%Let $\Hom(H_1,H_2)$ be the set of homomorphism from $H_1$ to $H_2$.
\end{definition}

\begin{definition}\label{def:vhf_valuedhyperfield}\cite[Definition 2.7]{Lee20}
Given $n\ge 1$, the {\em $n$th valued hyperfield} of $(K,\nu)$ is a hyperfield $(K/(1+\fm_{\nu}^{n}),\hp,\hm)$, denoted by $\H_{\nu,n}$, such that for each $a(1+\fm_{\nu}^{n}),b(1+\fm_{\nu}^n)\in K/1+\fm_{\nu}^n$,
	\begin{enumerate}
		\item $a(1+\fm_{\nu}^{n})\hm b(1+\fm_{\nu}^{n}):=ab(1+\fm_{\nu}^{n})$,
		\item $a(1+\fm_{\nu}^{n})\hp b(1+\fm_{\nu}^{n}):=\{(x+y)(1+\fm_{\nu}^{n})|\ x\in a(1+\fm_{\nu}^{n}),y\in b(1+\fm_{\nu}^{n})\}$, and
	\end{enumerate}
Conventionally, we write $0$ for $0(1+\fm_{\nu}^{n})$ and $1$ for $1(1+\fm_{\nu}^{n})$. The valuation $\nu$ of $K$ induces a map $\hn$ on $K/(1+\fm_{\nu}^{n})$ sending $a(1+\fm_{\nu}^{n})$ to $\nu(a)$. We call $(K/(1+\fm_{\nu}^{n}),\hp,\hm,\hn)$ the {\em $n$th valued hyperfield}.
\end{definition}
\noindent For $A\subset K$, let $\H_{\nu,n}(A):=\{a(1+\fm_{\nu}^{n}):\ a\in A\}$ and for $a\in K$, we write $[a]_n$ for $a(1+\fm_{\nu}^{n})\in \H_{\nu,n}$. Let $S_{\nu}$ be the set of units in the valuation ring and the zero, that is, $$S_{\nu}:=\mathcal O^{\times}\cup \{0\}=\{x\in K:\nu(x)=0\mbox{ or } x=0\}.$$  
%If there is no confusion, we omit $\nu$ in the subscript.

%\begin{remark}\label{rem:vhf_valuegroup}
%The quotient group $\H_{\nu,n}^{\times}/\H_{\nu,n}(\mathcal O^{\times})$ is an ordered group isomorphic to $\Gamma$.
%\end{remark}

\begin{definition}\label{def:over_p}
Let $(K,\nu)$ and $(L,\omega)$ be valued fields of mixed characteristic $(0,p)$. For $n\ge 1$, we say a homomorphism $f:\H_{\nu,n}\rightarrow \H_{\omega,n}$ is {\em over $p$} if $f([p]_n)=[p]_n$.
\end{definition}

\begin{fact}[{\cite[Lemma 3.1]{Lee20}}]\label{fact:addition on hyperfields}
Let $a,b\in K$ and $a_0,\ldots,a_k\in K$. Fix $n\ge 1$.
\begin{enumerate}
	\item $\bigcup [a]_{n}=a(1+\mathfrak{m}_{\nu}^n)=\{x\in K:\nu(x-a)\ge n+\nu(a)\}$.
	\item 
	$\bigcup [a]_n+_{\mathcal H} [b]_n=
	\begin{cases}
	(a+b)(1+\mathfrak{m}_{\nu}^n)&\mbox{if }\nu(a+b)-\min\{\nu(a),\nu(b)\}=0,\\
	(a+b)+\mathfrak{m}_{\nu}^{n+\min\{\nu(a),\nu(b)\}}&\mbox{if }0<\nu(a+b)-\min\{\nu(a),\nu(b)\}\le n,\\
	\mathfrak{m}_{\nu}^{n+\min\{\nu(a),\nu(b)\}}&\mbox{if } n<\nu(a+b)-\min\{\nu(a),\nu(b)\}.
	\end{cases}
	$
	\item $0\in\bigcup [a]_n+_{\mathcal H} [b]_n$ if and only if $\bigcup [a]_n+_{\mathcal H} [b]_n=\mathfrak{m}_{\nu}^{n+\min\{\nu(a),\nu(b)\}}$.
	\item $(a_0+\cdots +a_k)\in \sum^{\mathcal H}[a_i]_n$.
	\item Suppose $b\in \bigcup \sum^{\mathcal H}[a_i]_n$ and $a_0,\ldots,a_k\in R$. Then, $$b=(a_0+\cdots+a_k)+d$$ for some $d\in \mathfrak{m}_{\nu}^n$.
\end{enumerate}
\end{fact}

\begin{remark}\label{rem:description of H_n(S)}
For each $n\ge 1$, given $a,b\in S_{\nu}$, $$a\equiv b\pmod{\mathfrak{m}_{\nu}^n}\Leftrightarrow [a]_n=[b]_n.$$
Moreover, $$\mathcal H_{\nu,1}(S)\cong k.$$
\end{remark}
\begin{proof}
Fix $n\ge 1$. Given $a\in S$ with $a\neq 0$, we have
\begin{align*}
[a]_n&=a(1+\mathfrak{m}_{\nu}^n)\\
&=a+a\mathfrak{m}_{\nu}^n\\
&=a+\mathfrak{m}_{\nu}^n,
\end{align*}
and so for $a,b\in S_{\nu}$ with $a,b\neq 0$, $$[a]_n=[b]_n\Leftrightarrow a\equiv b\pmod{\mathfrak{m}^n}.$$ Also, if $a=0$ or $b=0$, and $a\equiv b\pmod{\mathfrak{m}_{\nu}^n}$, then $a=b=0$.\\

For the moreover part, consider a map $\varphi:\mathcal H_{\nu,1}(S_{\nu})\rightarrow k, [x]_1\mapsto x\pmod{\mathfrak{m}_{\nu}}$, which is a well-defined bijection. Also, by Fact \ref{fact:addition on hyperfields}, the map $\varphi$ is a field isomorphism. 
\end{proof}

Next, we give a choice of representatives in $\mathcal H_{\nu,n}(S_{\nu})$ of $K\nu$, analogous to analogous to Fact \ref{fact:uniqune representative}, and using the choice of representatives, we give a description of $\mathcal H_{\nu,n}$, analogous to Fact \ref{fact:generated by p-basis and representatives}.
\begin{remark}\label{rem:representatives on H_n} 
There is a unique choice of representatives $$\eta_{l+1}:k^{(p^l)}\rightarrow \mathcal H_{\nu,l+1}(S_{\nu})^{(p^l)},\ \alpha^{p^l}\mapsto [a^{p^{l}}]_{l+1}$$ where $a\in S_{\nu}$ with $\alpha=\res_{\nu}(a)$ for each $l\ge 0$. Note that $\eta_{l+1}(\alpha^{p^l})=[\lambda_{l+1}(\alpha^{p^l})]_{l+1}$ for $\alpha\in k$.
\end{remark}
\begin{proof}
Take $a,b\in S$ with $\res_{\nu}(a)=\res_{\nu}(b)(=:\alpha\in k)$ and so $a\equiv b\pmod{\mathfrak{m}_{\nu}}$. Then, $a^{p^l},b^{p^l}\in S_{\nu}$ and by Fact \ref{fact:uniquness of p-power representative}, $a^{p^l}\equiv b^{p^l}\pmod{\mathfrak{m}_{\nu}^{l+1}}$. So, we have
\begin{align*}
[a^{p^l}]_{l+1}&=a^{p^l}(1+\mathfrak{m}_{\nu}^{l+1})=a^{p^l}+\mathfrak{m}_{\nu}^{l+1}\\
&=b^{p^l}+\mathfrak{m}_{\nu}^{l+1}=b^{p^l}(1+\mathfrak{m}_{\nu}^{l+1})\\
&=[b^{p^l}]_{l+1},
\end{align*}
which implies the uniqueness of representative of $\alpha^{p^l}$ in $\mathcal H_{\nu,l+1}(S_{\nu})$.
\end{proof}

For model theory of valued hyperfields, we use a language $\mathcal L_{vhf}:=\{0,1,\cdot,+,|\}$ for valued hyperfields, where $0$ and $1$ are constant symbols, $\cdot$ is a binary function symbol, $+$ is a ternary relation, and $|$ is a binary relation. For a valued hyperfield $\mathcal H:=(H,\cdot,+,\nu)$, $1^H$ is interpreted as the identify of the multiplication, $0^{\mathcal H}$ as the identity of the addition, $\cdot^{\mathcal H}$ as the multiplication function on $H$. For $\alpha,\beta,\gamma\in H$, $(\alpha,\beta,\gamma)\in +^{\mathcal H}$ if and only if $\gamma\in \alpha+\beta$, and $(\alpha,\beta)\in |^{\mathcal H}$ if and only if $\nu(\alpha)\le \nu(\beta)$.

For valued hyperfields $(H_1,+_1,\cdot_1,0_1,1_1,\nu_1)$ and $(H_2,+_2,\cdot_2,0_2,1_2,\nu_2)$, a map $f:H_1\rightarrow H_2$ is called a {\em $\mathcal L_{vhf}^+$-homomorphism} if for any positive $\mathcal L_{vhf}$-formula $\theta(x)$ and for any $a\in H_1^{|x|}$, $$H_1\models \theta(a)\Leftrightarrow H_2\models \theta(f(a)).$$ An important point of the language $\mathcal L_{vhf}$ is that for finitely ramified valued fields, a $\mathcal L_{vhf}^+$-homomorphism over $p$ of higher valued hyperfields is a homomorphism over $p$ of higher valued hyperfields, which is crucially used in Subsections \ref{subsect:relative_existential_completness} and \ref{subsect:existential_closedness}.

\begin{remark}\label{rem:L_vhf homomorphism and vhf homomorphism}
For two valued hyperfields $(H_1,+_1,\cdot_1,0_1,1_1,\nu_1)$ and $(H_2,+_2,\cdot_2,0_2,1_2,\nu_2)$, let $f:H_1\rightarrow H_2$ be a $\mathcal L_{vhf}^+$-homomorphism such that there is $\alpha\in H_1$ such that $\nu_1(\alpha)$ and $\nu_2(f(\alpha))$ are positive, and the interval $(0_{\Gamma},\nu_1(\alpha)]$ in the value group is finite. Then, $f$ is a homomorphism between valued hyperfields.
\end{remark}
\begin{proof}
Since $f$ is a $\mathcal L_{val}^{+}$-homomorphism, it is enough to show that given $\alpha_1,\alpha_2\in H_1$, $$\nu_2(f(\alpha_1))\le \nu_2(f(\alpha_2))\Rightarrow \nu_1(\alpha_1)\le \nu_1(\alpha_2).$$ 
Suppose $\nu_2(f(\alpha_1))\le \nu_2(f(\alpha_2))$ and $\nu_1(\alpha_1)>\nu_1(\alpha_2)$. Since $\nu_1(\alpha_1)>\nu_1(\alpha_2)$ and $\nu_1(\alpha)$ is positive, there is $\beta\in H_1$ with $0_{\Gamma}<\nu_1(\beta)\le \nu_1(\alpha)$ such that $\nu_1(\alpha_2\cdot_1\beta)\le \nu_1(\alpha_1)$. Then, we have
\begin{align*}
&\nu_1(\alpha_2)<\nu_1(\alpha_2\cdot_1\beta)\le \nu_1(\alpha_1)\\
\Rightarrow&\nu_2(f(\alpha_2))\le \nu_2(f(\alpha_2))+\nu_2(f(\beta))\le \nu_2(f(\alpha_1))\le \nu_2(f(\alpha_2))\\
\Rightarrow&\nu_2(f(\beta))=0_{\Gamma}.
\end{align*}
Also, since the interval $(0_{\Gamma},\nu_1(\alpha)]$ is finite, there is a positive integer $n$ such that $\nu_1(\alpha)<n\nu_1(\beta)$ and so $0_{\Gamma}< \nu_2(f(\alpha))\le n\nu_2(f(\beta))=0_{\Gamma}$, which is a contradiction.
\end{proof}

\begin{example}
Let $(\mathbb Q,\nu_2)$ be a valued field  with the $2$-adic valuation and $\mathcal H_{\nu_2,1}$ be the first valued hyperfield of $(\mathbb Q,\nu_2)$. Let $(\mathbb F_2:=\{0,1\},\nu)$ be a valued hyperfield with a multivalued addition $+$ defined as follows: 
\begin{itemize}
	\item $0+0=\{0\}$,
	\item $1+0=0+1=\{1\}$, and
	\item $1+1=\{0,1\}$,
\end{itemize}
and the trivial valuation $\nu$, that is, $\nu(1)=0_{\Gamma}$ and $\nu(0)=\infty$. Consider a map $f:\mathcal H_{\nu_2,1}\rightarrow \mathbb F_2$ given as follows: For $\alpha\in \mathcal H_{\nu_2,1}$, $$f(\alpha)=0\Leftrightarrow\alpha=[0]_1.$$ Then, $f$ is a $\mathcal L_{val}^+$-homomorphism but not a homomorphism between valued hyperfields.
\end{example}

Also, the $\mathcal L_{vhf}$-structure of the $n$th valued hyperfield is uniformly existentially definable in $\mathcal L_{val}$, depending only on the residue characteristic $p$, the initial ramification index $e$, and $n$.
\begin{remark}\label{rem:existentially definability of the valued hyperfield}
Let $n\ge 1$ and let $(K,\nu)$ be a valued field of mixed characteristic $(0,p)$ having the initial ramification index $e$. The following relations on $K$ given as follows: Given $x,y,z\in K$, 
\begin{itemize}
	\item $[x]_{n}|[y]_{n}$,
	\item $[z]_{n}=[x]_{n}\cdot_{\mathcal H}[y]_{n}$, 
	\item $[z]_{n}\in [x]_{n}+_{\mathcal H}[y]_n$
\end{itemize}
are existentially definable in $\mathcal L_{val}$ depending only on $p$, $e$, and $n$.

So, given a $\mathcal L_{vhf}$-sentence $\varphi$, there is a $\mathcal L_{val}$-sentence $\tilde{\varphi}$ such that given a finitely ramified valued field $(L,\omega)$ of mixed characteristic $(0,p)$ having the initial ramification index $e$, $$\mathcal H_{\omega,n}\models \varphi\Leftrightarrow (L,\omega)\models \tilde{\varphi}.$$ Also, if $\varphi$ is a positive existential $\mathcal L_{vhf}$-sentence, then $\tilde \varphi$ can be taken as an existential $\mathcal L_{val}$-sentence.
\end{remark} 
\begin{proof}
Note that
$$[x]_{n}|[y]_{n}\Leftrightarrow \nu(x)\le \nu(y)$$
and so the binary relation $[x]_{n}|[y]_{n}$ is existentially definable.

We have
\begin{align*}
&[z]_{n}=[x]_{n}\cdot_{\mathcal H}[y]_{n}\\
\Leftrightarrow &\left(z\neq0\wedge \frac{xy}{z}\in 1+\mathfrak{m}_{\nu}^{n}\right)\vee\left( x=0\vee y=0\rightarrow z=0\right)\\
\Leftrightarrow &\left(z\neq0\wedge \frac{xy}{z}-1\in \mathfrak{m}_{\nu}^{n}\right)\vee\left( x=0\vee y=0\rightarrow z=0\right)\\
\Leftrightarrow &\exists \pi\left(z\neq 0\wedge \nu(\pi^e)=\nu(p)\wedge \nu(\frac{xy}{z}-1)\ge \nu(\pi^{n})\right)\vee\left( x=0\vee y=0\rightarrow z=0\right),
\end{align*}
and so the ternary relation $[z]_{n}=[x]_{n}\cdot_{\mathcal H}[y]_{n}$ is existentially definable.

Also, we have
\begin{align*}
&[z]_{n}\in [x]_{n}+_{\mathcal H}[y]_n\\
\Leftrightarrow&z\in x(1+\mathfrak{m}_{\nu}^n)+y(1+\mathfrak{m}_{\nu}^n)\\
\Leftrightarrow&\exists \pi,u,v\left(\nu(\pi^e)=\nu(p)\wedge \nu(u)\ge \nu(\pi^n)\wedge \nu(v)\ge \nu(\pi^n)\wedge \left(z=x(1+u)+y(1+v)\right)\right),
\end{align*}
and so the ternary relation $[z]_{n}\in [x]_{n}+_{\mathcal H}[y]_n$ is existentially definable.
\end{proof}

We recall some relationship between the higher valued hyperfield, the residue field, and the value group.
\begin{remark}\label{rem:relationship between vf rf vg}
Let $(K,\nu)\subseteq (L,\omega)$ be valued fields. Suppose $\mathcal H_{\nu,n}\preceq_{\exists} \mathcal H_{\omega,n}$ in $\mathcal L_{vfh}$ for some positive integer $n$. Then,
\begin{enumerate}
	\item $\omega L/\nu K$ is torsion-free; and
	\item $L\omega/K\nu$ is separable.
\end{enumerate}
\end{remark}
\begin{proof}
(1) To show that $\omega L/\nu K$ is torsion-free, take $\gamma \in \omega L$ such that for some positive $m>0$, $m\gamma\in \nu K$ and we will show that $\gamma \in \nu K$. There are $\alpha\in \mathcal H_{\omega,n}$ and $\beta\in \mathcal H_{\nu,n}$ such that $\nu_{H}(\alpha)=\gamma$ and $\nu_{\mathcal H}(\alpha^m)=\nu_{\mathcal H}(\beta)$. So, $$\mathcal H_{\omega,n}\models \exists x\left(\nu_{\mathcal H}(x^m)=\nu_{\mathcal H}(\beta)\right).$$ Since $\mathcal H_{\nu,n}\preceq_{\exists}\mathcal H_{\omega,n}$, there is $\alpha_0\in\mathcal H_{\nu,n}$ such that $$\mathcal H_{\omega,n}\models \nu_{\mathcal H}(\alpha_0^m)=\nu_{\mathcal H}(\beta),$$ which implies that $\gamma=\nu_{\mathcal H}(\alpha_0^m)$ is in $\nu K$. Thus, $\omega L/\nu K$ is torsion-free.\\

(2) To show that $L\omega/K\nu$ is separable, that is, $K\nu$ and $(L\omega)^{(p)}$ are linearly disjoint over $(K\nu)^{(p)}$, take $a_1,\ldots,a_k\in \mathcal O_{\nu}^{\times}$ such that $\res_{\omega}(a_1),\ldots,\res_{\omega}(a_k)$ are linearly dependent over $(L\omega)^{(p)}$ and we will show that $\res_{\omega}(a_1),\ldots,\res_{\omega}(a_k)$ are linearly dependent over $(K\nu)^{(p)}$. Suppose $\res_{\omega}(a_1),\ldots,\res_{\omega}(a_k)$ are linearly dependent over $(L\omega)^{(p)}$. Then, there are $b_1,\ldots,b_k\in \mathcal O_{\omega}$ such that $$L\omega\models \res_{\omega}(a_1)\res_{\omega}(b_1)^p+\cdots+\res_{\omega}(a_k)\res_{\omega}(b_k)^p=0.$$ Without loss of generality, we may assume that all of $b_i$'s are in $\mathcal O_{\omega}^{\times}$, that is, $\res_{\omega}(b_i)\neq 0$. Then, we have
\begin{align*}
&L\omega\models \sum_{i\le k }\res_{\omega}(a_i)\res_{\omega}(b_i)^p=0\\
\Leftrightarrow&\mathcal H_{\omega,n}\models 0\in\sum_{i\le k}^{\mathcal H}[a_i]_n[b_i]_n^p\\
\Rightarrow& \mathcal H_{\omega,n}\models \exists \beta_1,\ldots,\beta_k\left(\bigwedge_{i\le k}\nu_{\mathcal H}(\beta_i)=0_{\Gamma} \wedge 0\in \sum_{i\le k}^{\mathcal H}[a_i]_n\beta_i^p \right). 
\end{align*}
Since $\mathcal H_{\nu,n}\preceq_{\exists}\mathcal H_{\omega,n}$, there are $\beta_1',\ldots,\beta_k'\in \mathcal H_{\nu,n}$ such that $$\mathcal H_{\nu,n}\models \bigwedge_{i\le k}\nu_{\mathcal H}(\beta_i')=0_{\Gamma} \wedge 0\in \sum_{i\le k}^{\mathcal H}[a_i]_n(\beta_i')^p,$$ which implies that there are $b_1',\ldots,b_k'\in \mathcal O_{\nu}^{\times}$ such that
\begin{align*}
&\mathcal H_{\nu,n}\models 0\in\sum_{i\le k}^{\mathcal H}[a_i]_n[b_i']_n^p\\
\Rightarrow& K\nu\models \sum_{i\le k }\res_{\omega}(a_i)\res_{\omega}(b_i')^p=0,
\end{align*}
and so $\res_{\omega}(a_1),\ldots,\res_{\omega}(a_k)$ are linearly dependent over $(K\nu)^{(p)}$.
\end{proof}

\subsection{A uniformizer over the strict Cohen subring}\label{subsect:uniformizer}
Given a positive integer $e$, let $$d(e):=e(1+\nu_p(e))$$ where $\nu_p$ is the $p$-adic valuation on the field of rationals with the valuation group $\mathbb Z$.

\begin{fact}[Krasner's lemma]\label{fact:Krasner lemma}
Let $(K,\nu)$ be a henselian valued field and let $a,b\in K^{alg}$. We denote the unique extension of $\nu$ to $K^{alg}$ by $\nu$. Suppose $a$ is separable over $K$. Suppose that for all conjugates $a'(\neq a)$ over $K$, $$\nu(b-a)>\nu(a'-a).$$ Then, $K(a)\subseteq K(b)$.
\end{fact}

%\begin{fact}[{\cite[Chapter 2]{Lang94}}]\label{fact:Uniformizer and Eisenstein polynomial}
%Eisenstein polynomial and uniformizer
%\end{fact}

\begin{fact}[{\cite[Section 3]{LL21}}]\label{fact:value_of_different}
Let $(K,\nu)$ be a finitely ramified henselian $\mathbb Z$-valued field of mixed characteristic $(0,p)$ having the initial ramification index $e$. Let $F$ be the quotient field of a strict Cohen subring $R$ of $K$. Let $\pi$ be a uniformizer of $(K,\nu)$ and so it is a zero of an Eisenstein polynomial $P$ of degree $e$ over $R$ (c.f. \cite[Chapter 2]{Lang94}). We denote the unique extension of $\nu$ to $K^{alg}$ by $\nu$. Suppose $\nu$ is normalized, that is, $\nu(p)=1$. 
\begin{enumerate}
	\item For $M_{\nu}(\pi):=\max\{\nu(\pi-\pi'):\pi'\in K^{alg},P(\pi')=0,\ \pi'\neq\pi\}$, $$M_{\nu}\le \frac{d(e)}{e^2},$$ and for any uniformizer $\pi_0$ of $(K,\nu)$, $M_{\nu}(\pi)=M_{\nu}(\pi_0)(=:M_{\nu})$.
	\item For a finite extension $(K',\nu')$ of $K$ with the initial ramification index $e'$, if there is $b\in \mathcal O_{\nu'}$ with $P(b)\in (\mathfrak{m}_{\nu'})^{l+1}$ for some $l\ge e'\frac{d(e)}{e}$, then there is $c\in \mathcal O_{\nu'}$ with $P(c)=0$.
\end{enumerate}
\end{fact}
\begin{proof}
(1) It comes from \cite[Lemma 3.3, Lemma 3.9]{LL21}.
(2) Since $\nu$ is normalized, if $P(b)\in (\mathfrak{m}_{\nu'})^{l+1}$, then 
$$\nu(P(b))\ge \frac{l+1}{e'}> \frac{d(e)}{e}.$$ So, there is a zero $c$ of $P$ such that $$\nu(b-c)\ge \frac{d(e)}{e^2},$$ and $\nu(b-c)>M_{\nu}$. Since $\pi$ and $c$ are conjugate over $K$, for any conjugate $c'$ of $c$ over $F$, $$\nu(c-c')\le M_{\nu}<\nu(b-c),$$ and so by Krasner's lemma, $$c\in K(b)\subseteq K'.$$
\end{proof}

\section{Lifting of homomorphisms of hyperfields}\label{sect:lifting via valued hyperfield}
Let $(K,\nu)$ and $(L,\omega)$ be complete finitely ramified $\mathbb Z$-valued fields of mixed characteristic $(0,p)$ having the same initial ramification index $e$.
% Let $R_1$ and $R_2$ be valuation rings of $K_1$ and $K_2$ respectively and let $S_1:=R_1^{\times}\cup \{0\}$ and $S_2:=R_2^{\times}\cup\{0\}$.
%We will show that $\mathcal H_n(K_1)$ and $\mathcal H_n(K_2)$ are isomorphic over $p$ if and only if $K_1$ and $K_2$ are isomorphic, which partially generalizes \cite[Theorem 3.8 and Theorem 3.19]{Lee20} into the non-perfect residue case, concerning only isomorphisms.

\begin{remark}\label{rem:iso_hypevf_residue_iso}
Fix $l\ge 0$. Let $f:\mathcal H_{\nu,l+1}\rightarrow \mathcal H_{\omega,l+1}$ be a homomorphism over $p$. Then, the morphisms $f$ induces an embedding from $\mathcal H_{\nu,l+1}(S_{\nu})$ to $\mathcal H_{\omega,l+1}(S_{\omega})$ and induces an embedding from $K\nu(\cong \mathcal H_{\nu,1}(S_{\nu}))$ to $L\omega(\cong H_{\omega,1}(S_{\omega}))$, say $\varphi:K\nu\rightarrow L\omega$.

Also, given a tuple $B\subset S_{\nu}$ of representatives of a $p$-basis $\beta$ of $K\nu$, let  $B'\subset S_{\omega}$ be a tuple of representatives of $\varphi(\beta)$ with $f[\mathcal H_{\nu,l+1}(B)]=\mathcal H_{\omega,l+1}(B')$. In this case, we say $B$ and $B'$ are {\em compatible with $f$}.
%
%
%Suppose that $K_1$ is 
%\begin{itemize}
%	\item $k_2/\varphi(k_1)$ is a separable extension; and
%	\item $B'$ is contained in a strict Cohen subring $R_{un}'$ of $K_2$.
%\end{itemize}
%Then, by Fact \ref{fact:embedding lemma}, there is a unique embedding $\Phi:R_{un}(B)\rightarrow R_{un}'$ compatible with $\varphi$. In this case, we say the embedding $\Phi:R_{un}(B)\rightarrow R_{un}'$ is {\em compatible with $f$}.
\end{remark}

We start with the following lifting lemma to construct an embedding from an unramified complete $\mathbb Z$-valued field into a complete $\mathbb Z$-valued field via the higher valued hyperfield, analogous to \cite[Lemma 4.1]{ADJ24}.

\begin{lemma}\label{lem:lifting_homomorphism_from unramified fields}
Fix $l\ge 0$. Let $(K,\nu)$ and $(L,\omega)$ be complete $\mathbb Z$-valued fields of mixed characteristic $(0,p)$. Suppose $K$ is unramified. Then, we have  the following.
\begin{itemize}
	\item Let $f:\mathcal H_{\nu,l+1}\rightarrow \mathcal H_{\omega,l+1}$ be a homomorphism over $p$ and let $\varphi:K\nu\rightarrow L\omega$ be an embedding induced  by $f$.
	\item Let $B$ be a tuple of representatives of a $p$-basis $\beta$ of $k$, let $B'$ be a tuple of representatives of $\varphi(\beta)$, which are compatible with $f$, and let $\Phi:K\rightarrow L$ be an embedding of valued fields sending $B$ to $B'$, compatible with $\varphi$, which exists by Fact \ref{fact:lifting lemma}.
\end{itemize}
Then, $f$ is induced by $\Phi$, that is, for $a\in K$,  $$f([a]_{l+1})=[\Phi(a)]_{l+1}.$$
\end{lemma}
\begin{proof}
Take $a\in K$. By Fact \ref{fact:generated by p-basis and representatives}(2), for $s\ge 0$ with $\nu(p^s)<\l+1\le \nu(p^{s+1})$ and for $d\in \mathbb Z$ with $\nu(a)=\nu(p^d)$, there are
	\begin{itemize}
		\item a tuple $\bar b:=(b_1,\ldots,b_m)\in B^m$ for some $m\ge 1$; and
		\item tuples $\bar \alpha_{i}:=(\alpha_{i,I})_{I\in P_{m,l}}\in (K\nu)^{|P_{m,l}|}$ for $i\le s$;
	\end{itemize}
	such that $$a/p^d\equiv \sum_{i\le s}\left( \sum_{I\in P_{m,I}}\lambda_{l+1}(a_{i,I}^{p^l})\bar b^I\right)p^i \pmod{\mathfrak{m}_{\nu}^{l+1}},$$	
and so we have 
	$$\Phi(a/p^d)\equiv \sum_{i\le s}\left( \sum_{I\in P_{m,I}}\lambda_{l+1}(\varphi(a_{i,I})^{p^l})\Phi(\bar b)^I\right)p^i \pmod{\mathfrak{m}_{\omega}^{l+1}}.$$
Then, by Remark \ref{rem:representatives on H_n}, 
\begin{align*}
[a/p^d]_{l+1}&=[\sum_{i\le s}\left( \sum_{I\in P_{m,I}}\lambda_{l+1}(a_{i,I}^{p^l})\bar b^I\right)p^i]_{l+1}\\
&\in \sum_{i\le s}\left( \sum_{I\in P_{m,I}}[\lambda_{l+1}(a_{i,I}^{p^l})]_{l+1}[\bar b]_{l+1}^I\right)[p]_{l+1}^i.
\end{align*}
By taking $f$,
\begin{align*}
f([a/p^d]_{l+1})&\in \left(\sum_{i\le s}\left( \sum_{I\in P_{m,I}}f(\eta_{l+1}(a_{i,I}^{p^l}))f([\bar b]_{l+1})^I\right)f([p]_{l+1})^i\right)\\
&=\left(\sum_{i\le s}\left( \sum_{I\in P_{m,I}}\eta_{l+1}(\varphi(a_{i,I})^{p^l})[\Phi(\bar b)]_{l+1}^I\right)[p]_{l+1}^i\right).
\end{align*}
So, by Fact \ref{fact:addition on hyperfields} and Remark \ref{rem:description of H_n(S)}, 
\begin{align*}
f([a/p^d]_{l+1})&=[\sum_{i\le s}\left( \sum_{I\in P_{m,I}}\lambda_{l+1}(\varphi(a_{i,I})^{p^l})\Phi(\bar b)^I\right)p^i+d]_{l+1}\\
&=[\Phi(a/p^d)]_{l+1}
\end{align*}
for some $d\in \mathfrak{m}_{\omega}^{l+1}$. Thus, since $f$ is over $p$,
\begin{align*}
f([a]_{l+1})&=f([a/p^d]_{l+1}[p^d]_{l+1})=f([a/p^d]_{l+1})f([p^d]_{l+1})\\
&=[\Phi(a/p^d)]_{l+1}[p^d]_{l+1}=[\Phi(a/p^d)p^d]_{l+1}\\
&=[\Phi\left((a/p^d) p^d\right)]_{l+1}\\
&=[\Phi(a)]_{l+1}.
\end{align*}
\end{proof}

\begin{theorem}\label{theorem:lifting lemma}
Let $(K,\nu)$ and $(L,\omega)$ be complete $\mathbb Z$-valued fields of mixed characteristic $(0,p)$. Suppose $K$ is of the initial ramification index $e$. Let $n$ be a positive integer satisfying that $$\begin{cases}
n>0 &\mbox{if } p\not| e,\\
n>e^2M_{\nu}&\mbox{if } p|e. 
\end{cases}$$
\begin{enumerate}
	\item If there is a homomorphism over $p$ from $\mathcal H_{\nu,n}$ to $\mathcal H_{\omega,n}$, then there is an embedding from $K$ to $L$. Moreover, if $p$ does not divide $e$, then any homomorphism over $p$ from $\mathcal H_{\nu,n}$ to $\mathcal H_{\omega,n}$ is induced from a unique embedding from $K$ to $L$.
	
	\item If $L$ has the initial ramification index $e$, then $\mathcal H_{\nu,n}$ and $\mathcal H_{\omega,n}$ are isomorphic over $p$ if and only if $K$ and $L$ are isomorphic. 
\end{enumerate}
\end{theorem}
\begin{proof}
Let $f:\mathcal H_{\nu,l+1}\rightarrow \mathcal H_{\omega,l+1}$ and let $\varphi:K\nu\rightarrow L\omega$ be an embedding compatible with $f$. Let $R$ be a strict Cohen subring of $K$ and let $F$ be the quotient field of $R$, which is a complete unramified valued field. By Fact \ref{fact:lifting lemma}, there is an embedding $\Phi_0:F\rightarrow L$ compatible with $\varphi$ and by Lemma \ref{lem:lifting_homomorphism_from unramified fields}, for each $c\in F$, $f([c]_{l+1})=[\Phi_0(c)]_{l+1}$.\\

Suppose $K$ is tamely ramified, that is, $p\not|e$. Since $K$ is tamely ramified, there is a uniformizer $\pi$ of $K$ which is a zero of an Eisenstein polynomial $P(X)$ in $R[X]$ of the form $P(X)=X^e-pa$ for some $a\in R^{\times}$. Take $\pi_0\in L$ with $f([\pi]_{l+1})=[\pi_0]_{l+1}$ and so $[\pi_0]_{l+1}^e=[p\Phi(a)]_{l+1}$. Now consider a polynomial $Q(X)=X^e-(p\Phi(a))/(\pi_0)^e$ in $L[X]$. Then, since $Q(1)\in \mathfrak{m}_{\omega}^{l+1}$ and $Q'(1)=e\notin \mathfrak{m}_{\omega}$, by Hensel's lemma, there is a unique $b\in L^{\times}$ such that $b^e=(p\Phi(a))/(\pi_0)^e$ and $(b-1)\in \mathfrak{m}_{\omega}^{n}$. Let $\pi':=b\pi_0$ so that $(\pi')^e=p\Phi(a)$. Note that $[\pi']_{l+1}=[\pi_0]_{l+1}$ and $\pi'$ is such a unique zero of the polynomial $X^e-p\Phi(a)$. So, we have a unique isomorphism $\Phi:K\rightarrow L$ extending $\Phi_0\cup \{(\pi,\pi')\}$ and inducing $f$.\\

Suppose $K$ is wildly ramified, that is, $p|e$. We mimic the proof of \cite[Theorem 3.19]{Lee20} using Lemma \ref{lem:lifting_homomorphism_from unramified fields} instead of \cite[Lemma 3.6]{Lee20}. Since $\mathcal O_{\nu}$ is a totally ramified extension of $R$, there is a uniformizer $\pi$ of $K$ whose irreducible polynomial over $F$ is over $R$. Let $P(X)=X^e+a_{e-1}X^{e-1}+\ldots+a_0\in R[X]$ be the irreducible polynomial of $\pi$ over $F$ and let $\Phi_0(P)(X):=X^e+\Phi_0(a_{e-1})X^{e-1}+\ldots+\Phi_0(a_0)\in \Phi_0[R][X]$. Let $\pi_0\in \mathcal O_{\omega}$ such that $[\pi_0]=f([\pi])$. Then, we have that
\begin{align*}
0&=f([P(\pi)]_{l+1})\\
&=f([\pi^e+a_{e-1}\pi^{e-1}+\cdots+a_0]_{l+1})\\
&\in [\pi_0]_{l+1}^e+_{\mathcal H}f([a_{e-1}]_{l+1})[\pi_0]_{l+1}^{e-1}+_{\mathcal H}\cdots+_{\mathcal H}f([a_0]_{l+1})\\
&= [\pi_0]_{l+1}^e+_{\mathcal H}[\Phi_0(a_{e-1})]_{l+1}[\pi_0]_{l+1}^{e-1}+_{\mathcal H}\cdots+_{\mathcal H}[\Phi_0(a_0)]_{l+1}\\
&=[\pi_0^e]_{l+1}+_{\mathcal H}[\Phi_0(a_{e-1})\pi_0^{e-1}]_{l+1}+_{\mathcal H}\cdots+_{\mathcal H}[\Phi_0(a_0)]_{l+1}
\end{align*}
Since $\pi_0^e,\Phi_0(a_{e-1})\pi_0^{e-1},\ldots, \Phi_0(a_0)\in \mathcal O_{\omega}$, by Fact \ref{fact:addition on hyperfields}, we have $$0=\Phi_0(P)(\pi_0)+d$$ for some $d\in \mathfrak{m}_{\omega}^{n}$, and $\omega(\Phi_0(P)(\pi_0))\ge n/e> M(K_1)e$. By Fact \ref{fact:value_of_different}(2), there is $\pi'\in \mathcal O_{\omega}$ such that $\Phi_0(P)(\pi')=0$ and so we have an embedding $\Phi:K\rightarrow L$ extending $\Phi_0\cup \{(\pi,\pi')\}$.
\end{proof}

Using Theorem \ref{theorem:lifting lemma}, we will prove the embedding lemma \ref{lem:embedding over valued fields} over substructures, analogous to \cite[Lemma 4.10]{ADJ24}, via the valued hyperfield, and using this, we will provide elementary extensions and existential closedness in terms of valued hyperfields. The next lemma is analogous to \cite[Lemma 4.9]{ADJ24}.

\begin{lemma}\label{lem:unramified_embedding}
Let $(K_0,\nu_0)$ be an unramified $\mathbb Z$-valued field and $(K,\nu)$ a complete unramified extension of $(K_0,\nu_0)$ such that $K\nu/K_0\nu_0$ is separable. Let $(L,\omega)$ be a complete $\mathbb Z$-valued field.

Fix a positive integer $n$. Let $\Phi_0:(K_0,\nu_0)\rightarrow (L,\omega)$ be an embedding and let $\varphi_{\mathcal H}:\mathcal H_{\nu,n}\rightarrow \mathcal H_{\omega,n}$ be an embedding such that $\varphi_{\mathcal H}|_{\mathcal H_{\nu_0,n}}$ is induced from $\Phi_0$. Let $\beta$ be a $p$-basis of $K\nu$ over $K_0\nu_0$, $B$ be a lift of $\beta$ in $K$, and $B'$ a tuple of elements in $L$ such that $B$ and $B'$ are compatible with $\varphi_{\mathcal H}$.

Then, there is an embedding $\Phi:(K,\nu)\rightarrow (L,\omega)$ sending $B$ to $B'$, inducing $\varphi_{\mathcal H}$, and extending $\Phi_0$. 
\end{lemma}
\begin{proof}
Let $B_0$ be a tuple of representatives of a $p$-basis of $K_0\nu_0$ and let $B_0':=\Phi_0[B_0]$. Then, since $K\nu/K_0\nu_0$ is separable, $B_0\cup B$ is a tuple of reprsentatives of a $p$-basis of $K\nu$, and $B_0\cup B$ and $B_0'\cup B'$ are compatible with $\varphi_{\mathcal H}$. Since $(K,\nu)$ is unramified, by Theorem \ref{theorem:lifting lemma}(1), there is a unique embedding $\Phi:(K,\nu)\rightarrow (L,\omega)$ sending $B_0\cup B$ to $B_0'\cup B'$ and inducing $\varphi_{\mathcal H}$.

Also, since $\Phi$ and $\Phi_0$ are coincident on $B_0$, the proof of \cite[Lemma 4.9]{ADJ24} shows that $\Phi$ extends $\Phi_0$.
\end{proof}

\begin{lemma}\label{lem:embedding over valued fields}
Let $(K_0,\nu_0)$, $(K,\nu)$, and $(L,\omega)$ be complete finitely ramified $\mathbb Z$-valued fields of mixed characteristic $(0,p)$ having the same initial ramification index $e$ such that $K_0$ is a valued subfield of $K$ and $K\nu/K_0\nu_0$ is separable. Let $n$ be a positive integer satisfying that $$\begin{cases}
n>0 &\mbox{if } p\not| e,\\
n>e^2M_{\nu}&\mbox{if } p|e. 
\end{cases}$$

Let $\Phi_0:(K_0,\nu_0)\rightarrow (L,\omega)$ be an embedding and let $\varphi_{\mathcal H}^{0}:\mathcal H_{\nu_0,n}\rightarrow \mathcal H_{\omega,n}$ be an embedding induced from $\Phi_0$. Then, if there is an embedding $\varphi_{\mathcal H}:\mathcal H_{\nu,n}\rightarrow \mathcal H_{\omega,n}$ extending $\varphi_{\mathcal H}^{0}$, then there is an embedding $\Phi:(K,\nu)\rightarrow (L,\omega)$ extending $\Phi_0$.

Moreover, if $p$ does not divide $e$, then we can take such an embedding $\Phi$ inducing $\varphi_{\mathcal H}$.
\end{lemma}
\begin{proof}
We mimic the proof of \cite[Lemma 4.10]{ADJ24}. Let $F_0\subseteq K_0$ be a quotient field of a strict Cohen subring of $\mathcal O_{\nu_0}$, and let $F$ be a valued extension of $F_0$ which is complete and unramified with the residue field $K\nu$. By Lemma \ref{lem:unramified_embedding}, we may assume that $F$ is a subfield of $K$.

Since $[K:F]=[K_0:F_0]=e$, $K=FK_0$, and $K_0$ and $F$ are linearly disjoint over $F_0$. By Lemma \ref{lem:unramified_embedding}, we extend the embedding $F_0\rightarrow \Phi_0[F_0](\subseteq L)$ to an embedding $F\rightarrow L$ inducing $\varphi_{\mathcal H}|_{\mathcal H_{\nu,n}(F)}$. Since $K_0$ and $F$ are linearly disjoint over $F_0$, we have an embedding $\Phi$ defined on $FK_0(=K)$ extending $\Phi_0$, which is the desired one.\\

Suppose $p$ does not divide $e$. We will show that $\Phi$ induces $\varphi_{\mathcal H}$. Note that there is a uniformizer $\pi$ of $K_0$ which is a zero of Eisenstein polynomial $P(X)$ in $F_0[X]$ of the form $P(X)=X^e-a$ such that $K=F(\pi)$. By Theorem \ref{theorem:lifting lemma}(1), $\varphi_{\mathcal H}$ is induced from a unique embedding $\Phi':(K,\nu)\rightarrow (L,\omega)$ such that
\begin{itemize}
	\item for each $c\in F$, $\varphi_{\mathcal H}([c]_n)=[\Phi'(c)]_n$; and
	\item for $\pi':=\Phi'(\pi)=\Phi_0(\pi)$, $\varphi_{\mathcal H}([\pi]_n)=[\pi']_n$,
\end{itemize}
satisfied by the embedding $\Phi$, and so $\Phi$ is the unique embedding inducing $\varphi_{\mathcal H}$.
\end{proof}

\section{Ax-Kochen-Eroshov principle via hyperfields}\label{sect:AKE via hyperfields}
We first recall several facts on the coarsening of valuation. Let $(K,\nu)$ be a valued field, let $\nu^{\circ}$ be the finest coarsening of $\nu$, and let $\bar \nu$ be the valuation on the residue field $K\nu^{\circ}$. So, for the nontrivial smallest convex subgroup $(\nu K)^{\circ}$ of $\nu K$, $$\nu^{\circ}K\cong \nu K/(\nu K)^{\circ},\ \bar \nu K\nu^{\circ}\cong (\nu K)^{\circ}.$$ 

\begin{fact}\label{fact:coarsening and hyperfield}\cite[Remark 5.4 and Lemma 5.5]{Lee20}
Fix a positive integer $n$. Let $$\mathcal H_{\nu,n}^{\circ}:=\{\alpha\in \mathcal H_{\nu,n}:\nu_{\mathcal H}(\alpha)\in (\nu K)^{\circ}\}\cup \{0(\in \mathcal H_{\nu,n})\}.$$ Then, $\mathcal H_{\nu,n}^{\circ}$ forms a valued hyperfield and it is isomorphic to $\mathcal H_{\bar \nu,n}$ via an isomorphism $f:\mathcal H_{\nu,n}^{\circ}\rightarrow \mathcal H_{\bar \nu,n}$ given as follows: For  each $\alpha\in \mathcal H_{\nu,n}^{\circ}$, $$f(\alpha)=\begin{cases}
[\res_{\nu^{\circ}}(a)]_n&\mbox{if }\alpha=[a]_n,\\
0&\mbox{if } \alpha=0.
\end{cases}
$$
\end{fact}

\begin{remark}\label{rem:coarsening and hyperfield}
Let $(K,\nu)$ and $(L,\omega)$ be finitely ramified valued fields of mixed characteristic $(0,p)$. Given a positive integer $n$, a homomorphism over $p$ from $\mathcal H_{\nu,n}$ to $\mathcal H_{\omega,n}$ induces a homomorphism over $p$ from $\mathcal H_{\nu^{\circ},n}$ to $\mathcal H_{\omega^{\circ},n}$.
\end{remark}
\begin{proof}
Let $f:\mathcal H_{\nu,n}\rightarrow \mathcal H_{\omega,n}$ be a homomorphism over $p$. Since $(K,\nu)$ and $(L,\omega)$ are finitely ramified and $f$ is over $p$, given $\alpha\in \mathcal H_{\nu,n}$, $$\nu_{\mathcal H}(\alpha)\in (\nu K)^{\circ}\Leftrightarrow \omega_{\mathcal H}(f(\alpha))\in (\omega L)^{\circ},$$ which implies that $f[\mathcal H_{\nu,n}^{\circ}]\subseteq \mathcal H_{\omega,n}^{\circ}$. So, by Fact \ref{fact:coarsening and hyperfield}, $f$ induces a homomorphism $\bar f:\mathcal H_{\bar \nu,n}\rightarrow \mathcal H_{\bar \omega, n}$. 
\end{proof}

\subsection{Relative completness}
\begin{theorem}\label{theorem:relative completeness}
Let $(K,\nu)$ and $(L,\omega)$ be finitely ramified henselian valued fields of mixed characteristic $(0,p)$ having the same initial ramification index $e$. Let $n$ be a positive integer satisfying that $$\begin{cases}
n>0 &\mbox{if } p\not| e,\\
n>e^2M_{\nu}&\mbox{if } p|e. 
\end{cases}$$
Then, $$(K,\nu)\equiv (L,\omega)\Leftrightarrow (\mathcal H_{\nu,n},[p]_n)\equiv(\mathcal H_{\omega,n},[p]_n).$$
\end{theorem}
\begin{proof}
It is enough to show the right-to-left direction. We follow the proof of $(3)\Rightarrow (1)$ in \cite[Theorem 5.8]{Lee20}. Suppose $(\mathcal H_{\nu,n},[p]_n)\equiv(\mathcal H_{\omega,n},[p]_n)$. By the Keisler-Shelah theorem, after taking the ultraproduct properly, we may assume that $(K,\nu)$ and $(L,\omega)$ are $\aleph_1$-saturated and $(\mathcal H_{\nu,n},[p]_n)\cong(\mathcal H_{\omega,n},[p]_n)$.

Since $(K,\nu)$ and $(L,\omega)$ are $\aleph_1$-saturated, $(K\nu^{\circ},\bar \nu)$ and $(L\omega^{\circ},\bar \omega)$ are complete $\mathbb Z$-valued fields of the same initial ramification index $e$. Since $(\mathcal H_{\nu,n},[p]_n)\cong (\mathcal H_{\omega,n},[p]_n)$, by Remark \ref{rem:coarsening and hyperfield}, we have $(\mathcal H_{\nu^{\circ},n},[p]_n)\cong(\mathcal H_{\omega^{\circ},n},[p]_n)$. By Theorem \ref{theorem:lifting lemma}, we have $(K\nu^{\circ},\bar \nu)\cong (L\omega^{\circ},\bar \omega)$ and so $(K\nu^{\circ},\bar \nu)\equiv (L\omega^{\circ},\bar \omega)$. Therefore, by a standard argument (c.f. \cite[Theorem 5.8]{Lee20}), we have $(K,\nu)\equiv (L,\omega)$.
\end{proof}

\begin{corollary}\label{cor:transfer of decidability of full theories}
Let $(K,\nu)$ be finitely ramified henselian valued fields of mixed characteristic $(0,p)$ having the initial ramification index $e$. Let $n$ be a positive integer satisfying that $$\begin{cases}
n>0 &\mbox{if } p\not| e,\\
n>e^2M_{\nu}&\mbox{if } p|e. 
\end{cases}$$
Then, $\Th_{\mathcal L_{val}}(K,\nu)$ is decidable if and only if $\Th_{\mathcal L_{vhf}}(\mathcal H_{\nu,n},[p]_n)$ is decidable.
\end{corollary}
\begin{proof}
It is standard. Let $T_{p,e}$ be the $\mathcal L_{val}$-theory of all finitely ramified henselian valued fields of mixed characteristic $(0,p)$ having the initial ramification index $e$, which is computably axiomatizable. Let $T_{\mathcal H_{\nu,n}}$ be the $\mathcal L_{vhf}$-theory of $\mathcal H_{\nu,n}$ and let $\tilde{T}_{\mathcal H_{\nu,n}}:=\{\tilde{\theta}:\theta\in T_{vhf}\}$ so that given a finitely ramified valued field $(L,\omega)$ of mixed characteristic $(0,p)$ having the initial ramification index $e$, $$\mathcal H_{\omega,n}\models T_{vhf}\Leftrightarrow (L,\omega)\models \tilde{T}_{vhf}$$ by Remark \ref{rem:existentially definability of the valued hyperfield}. Consider the $\mathcal L_{val}$-theory $T:=T_{p,e}\cup \tilde{T}_{\mathcal H_{\nu,n}}$, which is the complete $\mathcal L_{val}$-theory of $(K,\nu)$ by Theorem \ref{theorem:relative completeness}. Since $T$ is complete and $T_{p,e}$ is computably axiomatizable, we have that $T$ is decidable if and only if $\tilde{T}_{\mathcal H_{\nu,n}}$ is decidable if and only if $T_{\mathcal H_{\nu,n}}$ is decidable.
\end{proof}

Since the residue field and the value group are interpretable in the $n$th valued hyperfield in $\mathcal L_{vhf}$ for any positive integer $n$, by \cite[Theorem 5.4]{ADJ24}, we have the AKE principle for relative model completeness relative to higher valued hyperfields.
\begin{theorem}\label{thm:model completeness}
Let $(K,\nu)\subseteq (L,\omega)$ be two finitely ramified henselian valued fields of the same initial ramification index $e$. If $\mathcal H_{\nu,n}\preceq \mathcal H_{\omega,n}$ in $\mathcal L_{vhf}$ for some positive integer $n$, $(K,\nu)\preceq (L,\omega)$ in $\mathcal L_{val}$.
\end{theorem}

\subsection{Relative existential completeness}\label{subsect:relative_existential_completness}
Given a valued field $(K,\nu)$, we write $\Th^{\exists}_{\mathcal L_{val}}(K,\nu)$ for the existential theory and $\Th^{\exists^+}_{\mathcal L_{vhf}}(\mathcal H_{\nu,n}(K),[p]_n)$ for the positive existential theory over $[p]_n$. We show the relative existential completeness relative to higher valued hyperfields, analogous to \cite[Corollary 5.6]{ADJ24}.
\begin{theorem}\label{theorem:relative existential completeness}
Let $(K,\nu)$ and $(L,\omega)$ be finitely ramified henselian valued fields of mixed characteristic $(0,p)$ having the same initial ramification index $e$. Let $n$ be a positive integer satisfying that $$\begin{cases}
n>0 &\mbox{if } p\not| e,\\
n>e^2M_{\nu}&\mbox{if } p|e. 
\end{cases}$$
Then, $$\Th^{\exists}_{\mathcal L_{val}}(K,\nu)\subseteq \Th^{\exists}_{\mathcal L_{val}}(L,\omega)\Leftrightarrow \Th^{\exists^+}_{\mathcal L_{vhf}}(\mathcal H_{\nu,n},[p]_n)\subseteq \Th^{\exists^+}_{\mathcal L_{vhf}}(\mathcal H_{\omega,n},[p]_n).$$
\end{theorem}
\begin{proof}
By Remark \ref{rem:existentially definability of the valued hyperfield}, the $\mathcal L_{vhf}$-structures on $\mathcal H_{\nu,n}$ and $\mathcal H_{\omega,n}$ are existentially $\mathcal L_{val}$-definable uniformly in the valued fields $(K,\nu)$ and $(L,\omega)$, and so the left-to-right direction holds.\\

It remains to prove the right-to-left direction. We mimic the proof of \cite[Proposition 5.5]{ADJ24} using Theorem \ref{theorem:lifting lemma} instead of \cite[Lemma 4.5]{ADJ24}. After taking elementary extensions, we may assume that $(K,\nu)$ is $\aleph_1$-saturated and $(L,\omega)$ is $|K|^+$-saturated. Since $\Th^{\exists^+}_{\mathcal L_{vhf}}(\mathcal H_{\nu,n},[p]_n)\subseteq \Th^{\exists^+}_{\mathcal L_{vhf}}(\mathcal H_{\omega,n},[p]_n)$, by saturation, there is an $\mathcal L_{vhf}^+$-homomorphism $\mathcal H_{\nu,n}\rightarrow \mathcal H_{\omega,n}$ over $p$, which is a homomorphism of valued hyperfields by Remark \ref{rem:L_vhf homomorphism and vhf homomorphism}. Also, by saturation, $(K\nu^{\circ},\bar \nu)$ and $(L\omega^{\circ},\bar \omega)$ are complete $\mathbb Z$-valued fields of the same initial ramification index $e$. Then, by Remark \ref{rem:coarsening and hyperfield}, the homomorphism from $\mathcal H_{\nu,n}\rightarrow \mathcal H_{\omega,n}$ over $p$ induces a homomorphism $\mathcal H_{\bar \nu, n}\rightarrow \mathcal H_{\bar \omega,n}$ over $p$ and by Theorem \ref{theorem:lifting lemma}, there is an embedding $K\nu^{\circ}\rightarrow L\omega^{\circ}$. So, we have that $\Th^{\exists}_{\mathcal L_{val}}(K\nu^{\circ},\bar \nu)\subseteq \Th^{\exists}_{\mathcal L_{val}}(L\omega^{\circ},\bar \omega)$.

By a standard argument (c.f. \cite[Lemma2.3]{AF16}), there are existentially closed embeddings $(K\nu^{\circ},\bar \nu)\rightarrow (K,\nu)$ and $(L\omega^{\circ},\bar \omega)\rightarrow (L,\omega)$ and so $$\Th^{\exists}_{\mathcal L_{val}}(K\nu^{\circ},\bar \nu)=\Th^{\exists}_{\mathcal L_{val}}(K,\nu),\ \Th^{\exists}_{\mathcal L_{val}}(L\omega^{\circ},\bar \omega)=\Th^{\exists}_{\mathcal L_{val}}(L,\omega).$$ Thus, we have
$$\Th^{\exists}_{\mathcal L_{val}}(K,\nu)=\Th^{\exists}_{\mathcal L_{val}}(K\nu^{\circ},\bar \nu)\subseteq \Th^{\exists}_{\mathcal L_{val}}(L\omega^{\circ},\bar \omega)=\Th^{\exists}_{\mathcal L_{val}}(L,\omega).$$
\end{proof}

\begin{corollary}\label{cor:Hilbert's 10th via the higher valued field}
Let $(K,\nu)$ be finitely ramified henselian valued fields of mixed characteristic $(0,p)$ having the initial ramification index $e$. Let $n$ be a positive integer satisfying that $$\begin{cases}
n>0 &\mbox{if } p\not| e,\\
n>e^2M_{\nu}&\mbox{if } p|e. 
\end{cases}$$
Then, $\Th_{\mathcal L_{val}}^{\exists}(K,\nu)$ is decidable if and only if $\Th_{\mathcal L_{vfh}}^{\exists^+}(\mathcal H_{\nu,n},[p]_n)$ is decidable.
\end{corollary}
\begin{proof}
It is standard and similar to the proof of Corollary \ref{cor:transfer of decidability of full theories}. Let $T_{p,e}$ be the $\mathcal L_{val}$-theory of all finitely ramified henselian valued fields of mixed characteristic $(0,p)$ having the initial ramification index $e$, which is computably axiomatizable. Let $T_{(K,\nu)}^{\exists}$ be the set of all existential $\mathcal L_{val}$-sentences in $\Th_{\mathcal L_{val}}^{\exists}(K,\nu)$ and negations of all existential $\mathcal L_{val}$-sentences not in $\Th_{\mathcal L_{val}}^{\exists}(K,\nu)$. Let $T_{\mathcal H_{\nu,n}}^{\exists^+}$ be the set of all positive existential $\mathcal L_{vhf}\cup\{[p]_n\}$-sentences in $\Th_{\mathcal L_{vhf}}^{\exists}(\mathcal H_{\nu,n},[p]_n)$ and negations of all positive existential $\mathcal L_{vhf}\cup\{[p]_n\}$-sentences not in $\Th_{\mathcal L_{val}}^{\exists}(\mathcal H_{\nu,n},[p]_n)$, and let $\tilde{T}_{\mathcal H_{\nu,n}}^{\exists^+}:=\{\tilde{\theta}:\theta\in T_{\mathcal H_{\nu,n}}^{\exists^+}\}$ so that given a finitely ramified valued field $(L,\omega)$ of mixed characteristic $(0,p)$ having the initial ramification index $e$, $$\mathcal (H_{\omega,n},[p]_n)\models T_{\mathcal H_{\nu,n}}^{\exists^+}\Leftrightarrow (L,\omega)\models \tilde{T}_{\mathcal H_{\nu,n}}^{\exists^+}$$ by Remark \ref{rem:existentially definability of the valued hyperfield}. Let $T_0:=T\cup T_{(K,\nu)}^{\exists}$ and $T_1:=T_{p,e}\cup \tilde{T}_{\mathcal H_{\nu,n}}^{\exists^+}$. 

We will show $T_0\models T_1$ and $T_1\models T_0$. Given a valued field $(L,\omega)$, if $(L,\omega)\models T_0$, then $(L,\omega)\equiv_{\exists}(K,\nu)$ and by Theorem \ref{theorem:relative existential completeness}, $(\mathcal H_{\omega,n},[p]_n)\equiv_{\exists^+}(\mathcal H_{\nu,n},[p]_n)$. So, $(L,\omega)\models T_1$, which implies $T_0\models T_1$. Conversely, if $(L,\omega)\models T_1$, then $(\mathcal H_{\omega,n},[p]_n)\equiv_{\exists^+}(\mathcal H_{\nu,n},[p]_n)$ and by Theorem \ref{theorem:relative existential completeness} again, $(L,\omega)\equiv_{\exists}(K,\nu)$. So, $(L,\omega)\models T_0$, which implies $T_1\models T_0$.

Thus, since $T_0\models T_1$ and $T_1\models T_0$, we have that $\Th_{\mathcal L_{val}}^{\exists}(K,\nu)$ is decidable if and only if $\{\tilde{\theta}:\theta\in \Th_{\mathcal L_{vfh}}^{\exists^+}(\mathcal H_{\nu,n},[p]_n)\}$ is decidable if and only if $\Th_{\mathcal L_{vfh}}^{\exists^+}(\mathcal H_{\nu,n},[p]_n)$ is decidable.
\end{proof}

\subsection{Existential closedness}\label{subsect:existential_closedness}
In this subsection, we aim to prove the AKE principle for existential closedness relative to higher valued hyperfields. We start with rewriting the embedding lemma \cite[Lemma 5.12]{ADJ24} using the valued hyperfield.
\begin{lemma}\label{lem:embedding lemma}
Let $(K,\nu)$ and $(L,\omega)$ be two extensions of a valued field $(K_0,\nu_0)$, all of which are henselian and finitely ramified of mixed characteristic $(0,p)$ having the same initial ramification index $e$. Suppose that
\begin{itemize}
	\item $\nu K/\nu_0 K_0$ is torsion-free;
	\item $K\nu/K_0\nu_0$ is separable; and
	\item $(K_0,\nu_0)$ and $(K,\nu)$ are $\aleph_1$-saturated, and $(L,\omega)$ is $|K|^+$-saturated.
\end{itemize}
Let $n$ be a positive integer satisfying that $$\begin{cases}
n>0 &\mbox{if } p\not| e,\\
n>e^2M_{\nu}&\mbox{if } p|e. 
\end{cases}$$
Then, if there is an $\mathcal L_{vhf}^+$-embedding $\mathcal H_{\nu,n}\rightarrow \mathcal H_{\omega,n}$ over $\mathcal H_{\nu_0,n}$, then there is an embedding $K\rightarrow L$ over $K_0$.
\end{lemma}
\begin{proof}
Let $\varphi_{\mathcal H}:\mathcal H_{\nu,n}\rightarrow \mathcal H_{\omega,n}$ be an $\mathcal L_{vhf}^+$-embedding over $\mathcal H_{\nu_0,n}$, which is an embedding of valued hyperfields by Remark \ref{rem:L_vhf homomorphism and vhf homomorphism}, and it induces an embedding $\varphi_{\Gamma}:\nu K\rightarrow \omega L$ over $\nu_0 K_0$. Also, by Lemma \ref{lem:embedding over valued fields} with the inclusion $(K_0,\nu_0)\rightarrow (L,\omega)$, there is an embedding $\Phi_{\circ}:(K\nu^{\circ},\bar \nu)\rightarrow (L\omega^{\circ},\bar \omega)$ over $(K_0\nu_0^{\circ},\bar \nu_0)$. From the proof of \cite[Lemma 5.12]{ADJ24}, we have the desired embedding $K\rightarrow L$ over $K_0$.

Indeed, let $\varphi_{\Gamma}^{\circ}:\nu^{\circ}K\rightarrow \omega^{\circ}L$ be an embedding induced from the embedding $\varphi_{\Gamma}$. Then, the proof of \cite[Lemma 5.12]{ADJ24} gives an embedding $\psi:K^{\times}/1+\mathfrak{m}_{\nu^{\circ}}\rightarrow L^{\times}/1+\mathfrak{m}_{\omega^{\circ}}$ to make the following diagram commute:
$$
\begin{tikzcd}
1 \arrow[r] & (K\nu^{\circ})^{\times} \arrow[r] \arrow[d, "\Phi_{\circ}"]& K^{\times}/1+\mathfrak{m}_{\nu^{\circ}} \arrow[r] \arrow[d, "\psi"] & \nu^{\circ}K \arrow[r] \arrow[d, "\varphi_{\Gamma}^{\circ}"] &0\\
1 \arrow[r] & (L\omega^{\circ})^{\times} \arrow[r]& L^{\times}/1+\mathfrak{m}_{\omega^{\circ}} \arrow[r] & \omega^{\circ}L \arrow[r] & 0
\end{tikzcd}
$$
via the identifications $\mathcal O^{\times}_{\nu^{\circ}}/1+\mathfrak{m}_{\nu^{\circ}}\cong (K\nu^{\circ})^{\times}$ and $\mathcal O^{\times}_{\omega^{\circ}}/1+\mathfrak{m}_{\omega^{\circ}}\cong (L\omega^{\circ})^{\times}$, and this gives an embedding $K\rightarrow L$ over $K_0$ inducing $\psi$, which is the desired one.
\end{proof}

Using Lemma \ref{lem:embedding lemma}, we have an AKE principle for relative existential closedness relative to higher valued hyperfields.
\begin{theorem}\label{theorem:relative existential closedness}
Let $(K,\nu)\subseteq (L,\omega)$ be finitely ramified henselian valued fields of mixed characteristic $(0,p)$ having the same initial ramification index $e$. Let $n$ be a positive integer satisfying that $$\begin{cases}
n>0 &\mbox{if } p\not| e,\\
n>e^2M_{\nu}&\mbox{if } p|e. 
\end{cases}$$
If $\mathcal H_{\nu,n}\preceq_{\exists} \mathcal H_{\omega,n}$ in $\mathcal L_{vhf}$, then $(K,\nu)\preceq_{\exists} (L,\omega)$ in $\mathcal L_{val}$.
\end{theorem}
\begin{proof}
We mimic the proof of \cite[Theorem 5.13]{ADJ24}. Taking ultrapowers properly, we may assume that $(K,\nu)$ and $(L,\omega)$ are $\aleph_1$-saturated, and consider an $|L|^+$-saturated elementary extension $(K^*,\nu^*)$ of $(K,\nu)$.

Since $\mathcal H_{\nu,n}\preceq_{\exists} \mathcal H_{\omega,n}$, by Remark \ref{rem:relationship between vf rf vg}, $\omega L/\nu K$ is torsion-free and $L\omega/K\nu$ is separable. Also, since $\mathcal H_{\nu,n}\preceq_{\exists} \mathcal H_{\omega,n}$ and $(K^*,\nu^*)$ is $|L|^+$-saturated, there is an $\mathcal L_{vhf}^+$-embedding $\mathcal H_{\omega,n}\rightarrow \mathcal H_{\nu^*,n}$ over $\mathcal H_{\nu,n}$, which is an embedding of valued hyperfields. Therefore, by Lemma \ref{lem:embedding lemma}, there is an embedding $(L,\omega)\rightarrow (K^*,\nu^*)$ over $(K,\nu)$, which implies that $(K,\nu)\preceq_{\exists} (L,\omega)$.
\end{proof}

\end{document}